\documentclass[11pt]{amsart}
\usepackage{amssymb, graphicx, color}
\usepackage{amsfonts}
\usepackage{amsthm}
\usepackage{enumerate}
\usepackage[mathscr]{eucal}
\usepackage{graphicx}

\usepackage[bookmarksnumbered, bookmarksopen, colorlinks, citecolor=blue, linkcolor=blue]{hyperref}

\allowdisplaybreaks

\newtheorem{thm}{Theorem}[section]
\newtheorem{defi}[thm]{Definition}
\newtheorem{cor}[thm]{Corollary}
\newtheorem{lem}[thm]{Lemma}

\newtheorem{prop}[thm]{Proposition}

\newtheorem{remark}[thm]{Remark}
\newtheorem{rk}[thm]{Remark}

\theoremstyle{definition}
\theoremstyle{remark}
\numberwithin{equation}{section}

\newcommand{\R}{{\mathbb R}}
\newcommand{\N}{{\mathbb N}}
\newcommand{\Z}{{\mathbb Z}}

\newcommand{\M}{{\mathcal M}}

\newcommand{\n}{{\mathcal N}}
\newcommand{\bs}{\begin{split}}
\newcommand{\es}{\end{split}}

\newcommand{\be}{\begin{eqnarray*}}
\newcommand{\ee}{\end{eqnarray*}}
\newcommand{\beq}{\begin{align}}
\newcommand{\eeq}{\end{align}}

\def\Q{\mathcal{Q}}
\def\1{\mathbf{1}}

\begin{document}

%
%
%
%
%
%
%
%
\setcounter{page}{1}
\title[Noncommutative maximal differential transforms]
{Noncommutative maximal differential transforms associated to averaging operators}

\author[S.  Liang]{Shaohong Liang}

\address{Department of Mathematics, Zhejiang University of Science and Technology, Hangzhou 310023, China}
\email{liangshaohongzk@163.com}

\author[Y. Wang]{Yu Wang}

\address{School of Mathematics and Big data, Jining University, Qufu 273155, China}

\email{202410110@jnxy.edu.cn}

\author[B. Xu]{Bang Xu}

\address{School of Mathematical Sciences, Xiamen University,  Xiamen 361005, China}

\email{bangxu@xmu.edu.cn}

\author[C. Zhang]{Chao Zhang$^\dag$}

\address{Department of Mathematics, Zhejiang University of Science and Technology, Hangzhou 310023, China}

\email{zaoyangzhangchao@163.com}


\thanks{B. Xu was supported by National Natural Science Foundation of China (Grant No.
12501172), and Fundamental Research Funds for the Central Universities (Grant No.
20720250061). C. Zhang was supported  by the National Natural Science Foundation of China (Grant No.  11971431).}

\thanks{$^\dag$ Corresponding Author.  }

\begin{abstract}
In this paper, we establish the noncommutative maximal weak type $(1,1)$ and strong type $(p,p)$ estimates for the family of operators $(T_N)_N$, defined   by
$$T_Nf=\sum_{k=N_1}^{N_2}\nu_{k}(M_{k}-\mathsf{E}_k)f,$$
where $M_k$ denotes the dyadic Hardy--Littlewood average operator, $\mathsf{E}_{k}$ is the conditional expectation with respect to the dyadic cubes of side-length $2^{-k}$, $N=(N_1,N_2)$ with $N_1<N_2$ and $(\nu_{k})\in\ell_{\infty}$. The main novelty of our approach is the development of a noncommutative Cotlar-type inequality for non-smooth kernels, a result that is new even in  classical harmonic analysis.
As an application, we obtain the boundedness theory of the noncommutative maximal  differential transforms for averaging operators.
\end{abstract}

\maketitle

\section{Introduction}
The study of differential transforms and their boundedness properties has a long history in classical analysis, particularly in probability theory, harmonic analysis and ergodic theory. Originally,  Burkholder~\cite{Bur66} considered the martingale transforms in martingale theory in 1966. Let $(\Omega, \mathcal F, P)$  be  a probability space , and $\nu_n$ be a $\mathcal F_n$-measurable function with $\mathcal F_0\subset \mathcal F_1\subset \cdots\subset \mathcal F$ being $\sigma$-fields such that $\{f_n, \mathcal F_n \}_{n\ge 1}$ is a martingale.  Burkholder proved that  the following transform
$$g_n=\sum_{k=1}^{n}v_k (f_k-f_{k-1})$$
construct a new martingale when the maximal function  $\nu^*(\omega)=\sup_n |\nu_n(\omega)|$ is finite.
More generally, let $(T_k)$ be a sequence of operators  and $(a_k)$ be an increasing sequence of positive real numbers. For a reasonable function $f$, the differential transform series is defined as
$$\sum_{k\in\Z}\nu_k(T_{a_{k+1}}-T_{a_k})f,$$
where $(\nu_{k})\in\ell_{\infty}$.
In \cite{JR02}, Jones and Rosenblatt  studied the $L_p$-boundedness  for such series in dynamical system. This program was subsequently extended by the last author and his collaborators \cite{ZMT21,ZT2021} associated with the Poisson and heat semigroups in harmonic analysis. Further developments, together with an extensive bibliography, can be found in \cite{LWZ23,Zhang2024} and the references therein.

The differential transform associated to averaging operators is defined for a reasonable function $f$ on $\R^d$ by
$$	Df=\sum_{k\in\Z}\nu_{k}(M_k -M_{k-1})f,$$
where $M_k$ denotes the dyadic Hardy-Littlewood averaging operator
$$M_k f(x)= \frac{1}{|B_k|} \int_{B_k}f(x+y)dy= \frac{1}{|B_k|} \int_{\R^d}f(y) \1_{B_k}(x-y)~dy \quad x\in\R^d,$$
and $B_k$ is the open ball centered at the origin with radius $r(B_k)=2^{-k}$.
A closely related operator is the difference between averaging operators and dyadic conditional expectations
\begin{align}\label{xu-zhang2026}
	Tf=\sum_{k\in\Z}\nu_{k}(M_{k}-\mathsf{E}_k)f,
\end{align}
where $\mathsf{E}_{k}$ denotes the $k$-th conditional expectation with respect to the $\sigma$-algebra  generated by the standard dyadic cubes of side length $2^{-k}$.
A significant observation is that the $L_p$-boundedness of $T$ immediately implies that of $D$ (see e.g. \cite{JR02}). Furthermore, it is noteworthy that the square function form of $T$ obtained by replacing $(\nu_k)$ with the Rademacher sequence $(\varepsilon_k)$ plays  an essential role in establishing variational inequalities for ergodic averages. For further background on this topic, the reader may refer to, for instance, \cite{Bou89, HM1, JKRW98, JRW03, JSW08, LeXu2} and references therein.

Over the past two decades, noncommutative analysis has developed a parallel theory of harmonic analysis, ergodic theory and martingale inequalities in the setting of von Neumann algebras.
Noncommutative $L_p$-spaces introduced by Segal, Dixmier, and others, provide a natural framework for extending classical results to operator-valued functions. Significant progress has been made in noncommutative Calder\'on-Zygmund theory \cite{C,CCP,HLX,JP1}, noncommutative ergodic theory \cite{HLW,JX}, noncommutative martingale transforms \cite{JSZ,Ran02}, among other areas.

In this spirit, the third author \cite{Xu2023} extended the uniform $L_p$-boundedness theory of
$T$ defined in (\ref{xu-zhang2026}) to the operator-valued setting, and then obtained  corresponding results for the differential transforms $D$ in conjunction with martingale transforms \cite{Ran02}. The main difficulty in the proof lies in establishing the weak type $(1,1)$ estimate of $T$, since the kernel associated with $T$ does not enjoy any regularity. This stands in contrast to methods used in \cite{C,MP,JP1}, which depend crucially on Lipschitz regularity condition (see also \cite{CCP,HLX} for cases with weaker regularity assumptions). We need to emphasize that the $L_p$-estimates of operators mentioned above, which concern infinite summations over $k\in\Z$, should be interpreted as a consequence of  the corresponding uniform boundedness for all finite summations through a standard approximation argument (see e.g. \cite[Section 6.A]{JMX}).

Nevertheless, the study of the boundedness theory of maximal differential transforms in the noncommutative setting remains largely unexplored. Although both weak type $(1,1)$ and strong type $(p,p)$ estimates are known for noncommutative maximal martingale transforms \cite{JSZ} and certain square functions \cite{HLX2024,HX}, the corresponding boundedness theory for the maximal differential operators has still not been established.
This paper is devoted to study  the maximal weak type $(1,1)$ as well as strong type $(p,p)$ estimates of the families $(T_N)_N$ and $(D_N)_N$ in the operator-valued setting.  To better state our results, we first introduce some necessary notations; more precise definitions are deferred to Section \ref{Sec:Pre}.
Let $(\mathcal{M}, \tau)$ be a noncommutative measure space, where $\mathcal{M}$ denotes a semifinite von Neumann algebra equipped with a normal semifinite faithful trace $\tau$. Define $\mathcal{N} = L_{\infty}(\mathbb{R}^{d}) \overline{\otimes} \mathcal{M}$ as the tensor von Neumann algebra endowed with the tensor trace
$\varphi = \int dx \otimes \tau,$
where $dx$ denotes the Lebesgue measure. For $0 < p \leq \infty$, we denote by $L_p(\mathcal{M})$ and $L_p(\mathcal{N})$ the noncommutative $L_p$-spaces associated with the pairs $(\mathcal{M}, \tau)$ and $(\mathcal{N}, \varphi)$, respectively.
Note that for $0 < p < \infty$, the space $L_p(\mathcal{N})$ can be identified with the Bochner $L_p$-space
$
L_{p}\bigl(\mathbb{R}^{d}; L_p(\mathcal{M})\bigr).
$

Let $\Z^2_{<}=\{N=(N_1,N_2):N_1<N_2\}$. For every $N=(N_1,N_2)\in\Z^2_{<}$, define
\begin{align}\label{XZmain1}
	T_Nf=\sum_{k=N_1}^{N_2}\nu_{k}(M_{k}-\mathsf{E}_k)f,
\end{align}
whenever $(\nu_{k})\in\ell_{\infty}$ and $f\in C_{c}^{\infty}(\R^{d})\otimes S_{\M}$, where $S_{\M}$ denotes the subset of $\M$ with $\tau$-finite support.
Our first result is stated as follows. We refer  to next section for the definition of noncommutative (weak) maximal norms and b.a.u. (an abbreviation for bilaterally almost uniformly) convergence.
\begin{thm}\label{maintheorem}
	Let $T_{N}$ be defined as (\ref{XZmain1}) with an arbitrary sequence $\nu=(\nu_{k})$ satisfying $\|\nu\|_{\ell_{\infty}}\leq1$. The following conclusions hold.	
\begin{itemize}
	\item [(i)] For $1<p<\infty$, $(T_{N}f)_{N\in\Z^2_{<}}$ is of strong type $(p,p)$. More precisely, there exists a positive constant $C_{d,p}$ depends only on dimension $d$ and $p$ such that for all $f\in L_p(\n)$,
	$$\|{\sup_{N\in\Z^2_{<}}}^+T_{N}f\|_p\leq C_{d,p}\|f\|_p.$$

	\item [(ii)] $(T_{N}f)_{N\in\Z^2_{<}}$ is of weak type $(1,1)$. More precisely, there exists a positive constant $C_{d}$ depends only on dimension $d$ such that for any $f\in L_1(\n)$ and $\lambda>0$, there is a projection $e\in\n$  satisfying
	$$\forall N\in\Z^2_{<},\ \|eT_Nfe\|_\infty\leq\lambda\ \text{and}\ \varphi(1-e)\leq C_d\frac{\|f\|_1}{\lambda}.$$
	
	\item [(iii)] For any $f\in L_p(\n)$ with $1\leq p<\infty$,
	$$T_Nf\ \text{ converges}\  \text{b.a.u.}\ \text{as}\ N_1\rightarrow-\infty,N_2\rightarrow\infty.$$
\end{itemize}	
\end{thm}


An immediate difficulty in showing Theorem \ref{maintheorem} is that the classical maximal function $\sup_{N}|T_Nf|$ is not well-defined in general whenever $(T_Nf)_N$ is a sequence of operators.  Indeed, the right formulation of noncommutative maximal (resp. weak) $L_p$-inequalities is accomplished through the introduction of vector-valued noncommutative (resp. weak) $L_p$-spaces $L_p(\M;\ell_\infty)$ (resp. $\Lambda_{p,\infty}(\M;\ell_\infty)$). Precise definitions are provided in the subsequent section.

For every $N=(N_1,N_2)\in\Z^2_{<}$, set
\begin{align}\label{1211}
D_Nf=\sum_{k=N_1}^{N_2}\nu_{k}(M_k -M_{k-1})f.
\end{align}
Then together with the boundedness theory  of noncommutative maximal martingale transforms proved in \cite[Theorem 1.2]{JSZ}, Theorem \ref{maintheorem} immediately implies the following result.
\begin{cor}\label{t7}
Let $D_N$ be defined as (\ref{1211}) with $\|\nu\|_{\ell_{\infty}}\leq1$. Then the following statements hold.	
\begin{itemize}
	\item [(i)] For $1<p<\infty$, $(D_{N}f)_{N\in\Z^2_{<}}$ is of strong type $(p,p)$.

	\item [(ii)] $(D_{N}f)_{N\in\Z^2_{<}}$ is of weak type $(1,1)$.
	
	\item [(iii)] For any $f\in L_p(\n)$ with $1\leq p<\infty$,
	$$D_Nf\ \text{ converges}\  \text{b.a.u.}\ \text{as}\ N_1\rightarrow-\infty,N_2\rightarrow\infty.$$
\end{itemize}	
\end{cor}
\begin{remark}
Theorem \ref{maintheorem} extends \cite[Theorem 3]{JR02} to the noncommutative setting.  Along with Corollary \ref{t7} and the transference techniques from \cite{HLW,HLX}, it also suggests further applications in other related topics such as noncommutative ergodic theory. However, we shall not pursue these directions in the present paper, as they lie beyond the main scope of our current work.
\end{remark}

With Theorem \ref{maintheorem} in hand, we now can present the proof of Corollary \ref{t7}. For each $N\in\Z^2_{<}$, we decompose
\begin{align*}
	D_Nf& =\sum_{k=N_1}^{N_2}\nu_{k}\Big((M_{k}-\mathsf{E}_{k})f+(\mathsf{E}_{k}-\mathsf{E}_{k-1})f+(\mathsf{E}_{k-1}-M_{k-1})f\Big)\\
	& =\sum_{k=N_1}^{N_2}\nu_{k}(M_{k}-\mathsf{E}_{k})f+\sum_{k=N_1}^{N_2}\nu_{k}(\mathsf{E}_{k}-\mathsf{E}_{k-1})f
	+\sum_{k=N_1}^{N_2}\nu_{k}(\mathsf{E}_{k-1}-M_{k-1})f.
\end{align*}
Note that the maximal weak type $(1,1)$ and strong type $(p,p)$ boundedness of the first and the third terms in the final expression above are a consequence of Theorem \ref{maintheorem} (for the third term, re-index $j=k-1$ to obtain $-\sum_{j=N_1-1}^{N_2-1}\nu_{j+1}(M_j-\mathsf{E}_j)f$, which is of the form $T_{(N_1-1,N_2-1)}f$ with the coefficient sequence $(\nu_{j+1})$ still satisfying $\|(\nu_{j+1})\|_{\ell_\infty}\le1$), while the middle term we refer to \cite[Theorem 1.2]{JSZ}. Thus Corollary \ref{t7} (i) and (ii) can be  proved by using the triangle (resp. quasi-triangle) inequality in $L_{p}(\n;\ell_\infty)$ (resp.  $\Lambda_{1,\infty}(\n;\ell_\infty)$). Finally, the b.a.u.\ convergence result can be obtained in a manner similar to that of Theorem \ref{maintheorem} (iii).

\medskip

Let us briefly outline the proof of Theorem \ref{maintheorem}. The pointwise convergence result-Theorem \ref{maintheorem} (iii) can
be obtained as a byproduct of the maximal  inequalities through a
standard verification. For
the sake of completeness, we shall present a proof
in Section \ref{Sec:convergence}. The deduction of (ii) from (i) relies on the noncommutative Calder\'on-Zygmund decomposition developed in \cite{CCP}. Hence, our main effort is devoted to establishing the strong type $(p,p)$ estimates of $(T_N)_N$, which in turn follow from the noncommutative Cotlar-type inequality stated below.
\begin{prop}\label{cotlar1}
	Let $T$ be defined in (\ref{xu-zhang2026}) and $T_{N}$ be defined as (\ref{XZmain1}) with $\|\nu\|_{\ell_{\infty}}\leq1$. Then for any $1<s<\infty$, $1<p<\infty$ and $f\in L_p(\n)$,
	\begin{align}\label{cotlar2}
\|{\sup_{N\in\Z^2_{<}}}^+T_{N}f\|_p\leq C_{d,s}\Big(\|{\sup_{n\in\Z}}^+\mathsf{E}_{n}(Tf)\|_{p}+\|{\sup_{k\in\Z}}^+(M_k f^s)^{\frac{1}{s}})\|_{p}\Big).
	\end{align}
\end{prop}
Proposition \ref{cotlar1} is substantially stronger than the classical Cotlar-type inequality of Jones and Rosenblatt \cite[Lemma 2]{JR02}, which asserts the existence of a constant $C>0$ such that
	\begin{align}\label{cotlar3}
	{\sup_{N\in\Z^2_{<}}}|T_{N}f(x)|\leq C\Big(\big(M(|Tf|^2)(x)\big)^{\frac{1}{2}}+\big(M(f^2)(x)\big)^{\frac{1}{2}}\Big),
\end{align}
where $Mf$ denotes the Hardy-Littlewood maximal function.  As is well known, \eqref{cotlar3} suffices only for $2<p<\infty$ and a weak type $(2,2)$ bound. In the commutative case, the remaining range $1< p\leq2$ can be obtained by Marcinkiewicz interpolation together with the weak $(1,1)$ estimate. However, this interpolation argument is no longer valid in the noncommutative setting since the noncommutative Marcinkiewicz interpolation theorem fails for families of non-positive operators \cite{CR}. Therefore, a genuinely new approach is required to obtain the full range $1<p<\infty$. Our proof of \eqref{cotlar2} exploits the geometric structure of the kernel and introduces a novel operator decomposition. In fact, the resulting Proposition \ref{cotlar1} is novel even in the commutative case.


The rest of the paper is organized as follows.  Section \ref{Sec:Pre} recalls the necessary background on (vector-valued) noncommutative $L_p$-spaces, as well as the noncommutative Calder\'on-Zygmund decomposition. In Section \ref{Sec:3}, we prove Proposition \ref{cotlar1} and Theorem \ref{maintheorem} (i). The proof of Theorem \ref{maintheorem} (ii) will be presented in Section \ref{Sec:4}. Section \ref{Sec:convergence} provides the b.a.u.\ convergence argument.

\medskip

\textbf{Notation:} Throughout this paper, we denote by $C$ a positive constant whose value may vary from line to line. The notation $X \lesssim Y$ means that there exists a constant $C > 0$, independent of essential parameters, such that $X \leq C Y$. Similarly, $X \approx Y$ signifies that both $X \lesssim Y$ and $Y \lesssim X$ hold simultaneously.
\section{Preliminaries}\label{Sec:Pre}
\subsection{Noncommutative $L_{p}$-spaces}
Let $\M$ be a semifinite von Neumann algebra endowed with a normal semifinite faithful trace $\tau$.
Denote by $\M_{+}$ the positive cone of $\M$ and define
$\mathcal{S_{\M+}}$ as the set of all $x\in\M_{+}$ such that $\tau(\mathrm{supp}\ x)<\infty$, where $\mathrm{supp}\ x$ is the support of $x$.
Let $\mathcal{S_{\M}}$ be the linear span of $\mathcal{S_{\M+}}$. Then  $\mathcal{S_{\M}}$ becomes a weak-${*}$-dense $*$-subalgebra of $\M$.
For $0<p<\infty$, we set
$$\|x\|_{p}=\big(\tau(|x|^p)\big)^{1/p},\ \ x\in\mathcal{S}_{\M},$$
where $|x|=(x^{*}x)^{1/2}$. The noncommutative $L_{p}$-space associated with $(\M,\tau)$, denoted by $L_{p}(\M)$, is the completion of $(\mathcal{S_{\M}},\|\cdot\|_{p})$.
For convenience we identify $L_{\infty}(\M)$ with $\M$ equipped with the operator norm $\|\cdot\|_{\M}$. Finally, let $L_{p}(\M)_{+}$ denote the positive part of $L_{p}(\M)$. We refer to \cite{FK,P2} for more information about the noncommutative $L_{p}$-spaces.

\subsection{Noncommutative $\ell_\infty$-valued $L_p$-spaces} The $\ell_\infty$-valued noncommutative $L_{p}$-spaces were introduced by Pisier \cite{pisier} and Junge \cite{J1}. Let $I$ be an index set. Given $1\le p\leq\infty$,
we define $L_p(\mathcal {M};\ell_\infty(I))$ as the space of all sequences $x=(x_k)_{k\in I}$
in $L_p(\mathcal {M})$ that admit a factorization of the form
$$x_k=ay_kb,\ \ \ k\in I,$$
where $a,b\in L_{2p}( \mathcal {M})$  and  $y=(y_k)_{k\in I}$ a bounded sequence in $\M$.
The norm of $x$ in $L_p(\mathcal {M};\ell_\infty(I))$ is given by
$$
\|x\|_{L_p(\mathcal {M};\ell_\infty(I))}=\inf\left\{\big\|a\big\|_{{2p}}\sup_{k\in I}
\big\|y_k\big\|_\infty\big\|b\big\|_{{2p}}\right\},$$
where the infimum is taken over all such  factorizations. As usual, we  also write 
$$\|{\sup_{k\in I}}^+x_k\|_p:=\|x\|_{L_p(\mathcal {M};\ell_\infty(I))}.$$

The following property can be found in \cite[Remark 4.1]{CXY13}.
\begin{rk}\label{rk:MaxFunct}
	Let $x=(x_k)_{k\in I}$ be a sequence of selfadjoint operators in $L_p(\M)$.
	Then $x\in L_p(\M;\ell_\infty(I))$ iff there exists $a\in
	L_p(\M)_+$ such that $-a\leq x_k \le a$ for all $k\in I$ and
	$$\big \|(x_k)_{k\in I} \big \|_{L_p(\mathcal {M};\ell_\infty(I))}=\inf\big\{\|a\|_p\;:\; a\in L_p(\M)_{+},\; -a\leq x_k\le a,\;\forall\; k\in I\big\}.$$
\end{rk}

We now introduce the weak maximal quasi-norm (see e.g.\ \cite{HLX}), which is the noncommutative analogue of the weak
$L_{p}$-norm of a maximal function. Given $1\leq p< \infty$,  $\Lambda_{p,\infty}(\M;\ell_{\infty}(I))$ is defined as the space of all sequence $(x_k)_{k\in I}$ in weak $L_{p}$-space such that the following quasi-norm is finite:
\begin{align}\label{weakm}
	\|(x_k)_{k\in I}\|_{\Lambda_{p,\infty}(\M;\ell_{\infty}(I))}
	=\sup_{\lambda>0}\lambda\inf_{e\in\mathcal{P}(\M)}
	 \Big\{\big(\tau(e^{\perp})\big)^{\frac{1}{p}}:\|ex_ke\|_{\infty}\leq\lambda\ \mbox{for\ all}\ k\in I\Big\},
\end{align}
where $\mathcal{P}(\M)$ denotes the set of all projections in $\M$.
As in Remark~\ref{rk:MaxFunct}, the quasi-norm admits an equivalent description (see \cite[Remark 2.3]{HLX}).
\begin{rk}\label{weakm4}
	Let $x=(x_k)_{k\in I}$ be a sequence of selfadjoint operators in $L_{p,\infty}(\M)$. Then
	\begin{align*}\label{weakm1}
		\|(x_k)_{k\in I}\|_{\Lambda_{p,\infty}(\M;\ell_{\infty}(I))}=\sup_{\lambda>0}\lambda\inf_{e\in\mathcal{P}(\M)}
		 \Big\{\big(\tau(e^{\perp})\big)^{\frac{1}{p}}:-\lambda\leq ex_ke\leq\lambda\ \mbox{for\ all}\ k\in I\Big\}.
	\end{align*}
\end{rk}

\subsection{Almost uniform convergence} The following noncommutative analogue of the usual almost everywhere convergence was introduced by Lance \cite{Lan76}. Let $L_0(\M)$ be the $*$-algebra of $\tau$-measurable operators (see e.g. \cite{FK} for the precise definition).
\begin{defi}
Let $\mathcal {M}$ be a semifinite von Neumann algebra equipped with a normal semifinite faithful trace $\tau$. Let $x_k,x\in L_0(\M)$. The sequence
$(x_k)$ is said to converge bilaterally almost uniformly (b.a.u. in short) to $x$ if for any $\delta>0$, there is a
	projection $e\in \M$ such that
	$$\tau(e^\bot)<\delta \ \ \  \mbox{and} \ \ \ \lim _{k\rightarrow\infty}\|e(x_k-x)e\|_\infty=0.$$
\end{defi}
Note that in the commutative setting, the convergence defined above coincides with the usual almost everywhere convergence by virtue of Egorov theorem. 

\subsection{Noncommutative Calder{\'o}n-Zygmund decomposition}
In this subsection, we recall the noncommutative Calder{\'o}n-Zygmund decomposition based on noncommutative martingale theory and  Cuculescu's construction \cite{Cuc}.
Recall that a dyadic cube in $\R^{d}$ is the set
$$[2^{k}m_{1},2^{k}(m_{1}+1))\times\cdot\cdot\cdot\times[2^{k}m_{d},2^{k}(m_{d}+1)),$$
where $k,m_{1},\cdot\cdot\cdot,m_{d}\in\Z$.
Let $\Q$ denote the family of all dyadic cubes, and for $k \in \Z$ write $\Q_k=\{Q\in\Q:\ell(Q)=2^{-k}\}$. The $\sigma$-algebra generated by $\Q_k$ is denoted by $\sigma_k$ and set $\mathcal N_k=L_\infty(\mathbb R^d,\sigma_k,dx)\overline{\otimes}\M$. Then $(\mathcal{N}_k)_{k \in \Z}$ is an increasing filtration of von Neumann subalgebras whose union is weak-$*$-dense in $\mathcal N$. The associated conditional expectations are given by
$$\mathsf{E}_k(f):=f_{k}=\sum_{Q \in \Q_k}^{\null} f_Q \chi_Q,\;\forall f\in L_1(\mathcal N),$$
where $\chi_Q$ is the characteristic function of $Q$ and 
$f_Q = \frac{1}{|Q|} \int_Q f(y) \, dy.$

Consider
$$\mathcal N_{c,+} =\Big\{
f: \R^d \to \M\cap L_1(\M) \, \big| \ f \geq0, \
\overrightarrow{\mathrm{supp}} \hskip1pt f \ \ \mathrm{is \
compact} \Big\},$$
where
$\overrightarrow{\mathrm{supp}}$ stands for the support of $f$ as an
operator-valued function in $\R^d$.
It is not difficult to see that $\mathcal N_{c,+}$ is dense in $L_1(\mathcal N)_{+}$. Moreover, given $f \in \mathcal{N}_{c,+}$ and $\lambda>0$, there exists $m_{\lambda}(f)\in\Z$ such that $f_{k}\leq\lambda\1_{\mathcal{N}}$ for all $k\leq m_{\lambda}(f)$ (see \cite[Lemma 3.1]{JP1}), where $\1_{\mathcal{N}}$ denotes the unit
elements in $\mathcal{N}$. We will need
the following modified Cuculescu's theorem \cite[Lemma 3.1]{JP1}.
\begin{lem}\label{cucu}
Let $f \in \mathcal{N}_{c,+}$ and $\lambda>0$.
Then there exists a decreasing sequence of projections $(q_k)_{k\in\Z}$ defined by
$$q_k=\begin{cases} \1_\mathcal{N} & \mbox{if} \ k\leq m_{\lambda}(f),
	\\ \chi_{(0,\lambda]}(q_{k-1} f_k q_{k-1}) & \mbox{if} \ k > m_{\lambda}(f),
\end{cases}$$
with the following properties hold:
\begin{itemize}
\item [(i)] $q_k$ commutes with $q_{k-1} f_k
q_{k-1}$ for each $k\in\Z$;

\item [(ii)] $q_k$ belongs to $\mathcal{N}_k$ and $q_k f_k q_k \le \lambda \hskip1pt
q_k$ for each $k\in\Z$;

\item [(iii)] $\varphi \Big(
\mathbf{1}_\mathcal{N} - \bigwedge_{k \in\Z} q_k \Big) \le
\frac{\|f\|_1}{\lambda}.$
\end{itemize}
\end{lem}
Due to the atomic nature of filtrations, the projection $q_k$ constructed in Lemma \ref{cucu} admits the following expression
$$q_{k}=\sum_{Q\in\Q_{k}}q_{Q}\chi_{Q},$$
for each $k\in\Z$, and $q_{Q}$ are projections in $\M$ such that
$$q_{Q}=\begin{cases} \1_\M & \mbox{if} \ k \leq
m_{\lambda}(f),
\\ \chi_{(0,\lambda]} \big( q_{\widehat{Q}} f_Q q_{\widehat{Q}}\big) & \mbox{if} \ k > m_{\lambda}(f),
\end{cases}$$
where $\widehat{Q}$ is the dyadic father of $Q$. 
For $k\in\Z$, define
\begin{equation}\label{czd1}
p_k=q_{k-1}-q_k=\sum_{Q\in\Q_{k}}(q_{\widehat{Q}}-q_{Q})\chi_{Q}\triangleq\sum_{Q\in\Q_{k}}p_Q\chi_{Q},
\end{equation}
so that
$$\sum_{k \in \Z} p_k = \1_\mathcal{N} - q =q^\perp\quad \mbox{with} \quad q = \bigwedge_{k \in
\Z} q_k.$$

The noncommutative Calder{\'o}n-Zygmund decomposition is stated as follows \cite{CCP}.
\begin{thm}\label{czdecom}
Given $f\in\mathcal N_{c,+}$ and $\lambda>0$. Let $(q_k)_{k\in\Z}$ and $(p_k)_{k\in\Z}$ be two sequences of projections constructed above. Define a projection $\zeta\in \mathcal N$ by
\begin{equation}\label{czd9}
\zeta = \big(\bigvee_{Q\in \Q} p_Q\chi_{5Q}\big)^{\bot},
\end{equation}
where $5Q$ denotes the cube with the same center as $Q$ and $\ell(5Q)=5\ell(Q)$.
Then there exists a decomposition
\begin{equation}\label{czd76789}
f = g + b+h
\end{equation}
satisfying the following properties:
\begin{enumerate}[(1)]
\item [(i)]$\varphi(\mathbf 1_\mathcal{N}-\zeta) \leq 5^{d}\dfrac{\|f\|_1}{\lambda}$;
\item[(ii)]$g=qfq+\sum_{k\in\Z}p_{k}f_{k}p_{k}$ satisfies
$\| g\|_1 \le\|f\|_1 \quad \mbox{and} \quad \|
g\|_\infty \le 2^{d} \lambda$;
\item[(iii)]$b=\sum_{k\in\Z}b_{k}$ with
\begin{equation}\label{czd6}
b_{k}=p_k (f-f_{k}) p_{k}.
\end{equation}
Each $b_{k}$ satisfies two cancellation conditions: for $Q\in \Q_{k}$,
     $\int_Q b_{k} = 0;$
    and for all $x,y\in\R^{d}$ such that $y \in 5Q_{x,k}$, $\zeta(x)b_{k}(y)\zeta(x) = 0$, where $Q_{x,k}$ is the unique cube in $\Q_k$ containing $x$;
\item[(iv)]$h=\sum_{k\in\Z}h_{k}$ with
\begin{equation}\label{czd66}
	h_{k}=p_k (f-f_{k}) q_{k}+q_{k}(f-f_{k})p_k.
\end{equation}
Moreover, each $h_{k}$ satisfies the same cancellation conditions as $b_k$ described above.
\end{enumerate}
\end{thm}

\section{Proof of Proposition \ref{cotlar1} and Theorem \ref{maintheorem} (i)}\label{Sec:3}

In this section, we establish Proposition \ref{cotlar1} and the strong type
$(p,p)$ estimate for $(T_N)_{N\in\Z^2_{<}}$ stated in Theorem \ref{maintheorem} (i). We begin with several auxiliary lemmas. The first one is taken from  \cite[Theorem 1.1]{Xu2023}.
\begin{lem}\label{xu2023}
	Let $T$ be defined as (\ref{xu-zhang2026}) with an arbitrary sequence $\nu=(\nu_{k})$ satisfying $\|\nu\|_{\ell_{\infty}}\leq1$. Then for any $f\in L_p(\n)$ with $1<p<\infty$,  there exists a positive  constant $C_{d,p}$ such that
	$$\|Tf\|_{p}\leq C_{d,p}\|f\|_p.$$
\end{lem}

The following operator-valued version of the H\"{o}lder inequality is well-known to experts (see e.g. \cite[Lemma 2.4]{HLRX}).
\begin{lem}\label{holderinequality}
	Let $(\Omega,\mu)$ be a measure space. Then for $1< p<\infty$ and let $p^{\prime}$ be the conjugate of $p$, that is $1/ p^\prime +1/p =1$,
	\begin{equation}\label{maximal132}
		 \int_{\Omega}f(x)g(x)d\mu(x)\leq\Big(\int_{\Omega}f(x)^{p}d\mu(x)\Big)^{\frac{1}{p}}
		 \Big(\int_{\Omega}g(x)^{p^{\prime}}d\mu(x)\Big)^{\frac{1}{p^{\prime}}},
	\end{equation}
	where $f:\Omega\rightarrow L_{1}(\mathcal{M})+L_{\infty}(\mathcal{M})$ and $g:\Omega\rightarrow \mathbb{C}$ are positive functions such that all members of inequality (\ref{maximal132}) make sense.
\end{lem}
Now we are ready to prove Proposition \ref{cotlar1}
\begin{proof}[Proof of Proposition \ref{cotlar1}.]
By decomposing $f = f_{1}-f_{2} +i(f_{3}-f_{4})$ with positive operator-valued functions $f_{j}$ satisfying $\|f_{j}\|_{1}\leq\|f\|_{1}$ for all $j\in\{1,2,3,4\}$, we may assume that $f$ is positive. For the same reason, we may assume the sequence $(\nu_k)$ is positive. In the following fix $1<s<\infty$ and $1<p<\infty$. For each $n\in\Z$, define
$$ S_nf=\sum_{k\le n}\nu_k(M_{k}-\mathsf{E}_k)f.$$
Then it follows from the triangle inequality in $L_{p}(\n;\ell_\infty)$ that
$$\|{\sup_{N\in\Z^2_{<}}}^+T_{N}f\|_{p}\leq2\|{\sup_{n\in\Z}}^+S_{n}f\|_{p}.$$
Hence, it suffices to show
\begin{align}\label{XZmain2}
\|{\sup_{n\in\Z}}^+S_{n}f\|_{p}\leq C_{d,s}\Big(\|{\sup_{n\in\Z}}^+\mathsf{E}_{n}(Tf)\|_{p}+\|{\sup_{k\in\Z}}^+(M_k f^s)^{\frac{1}{s}}\|_{p}\Big).
\end{align}
Recall that
$\mathsf{E}_n\mathsf{E}_k=\mathsf{E}_k\mathsf{E}_n=\mathsf{E}_{\min\{n,k\}}.$ This implies that
\begin{align*}
	\mathsf{E}_nTf
	&=
	\sum_{k\le n}\nu_k(\mathsf{E}_nM_k-\mathsf{E}_k)f
	+
	\sum_{k>n}\nu_k\mathsf{E}_n(M_k-I)f .
\end{align*}
Subtracting above identity from the definition of $S_n$ yields
\begin{equation*}\label{eq:cotlar-decomposition}
	S_nf=\mathsf{E}_nTf+R_n^{(1)}f+R_n^{(2)}f,
\end{equation*}
where
\begin{equation*}\label{eq:R1}
	R_n^{(1)}f
	=
	\sum_{k\le n}\nu_k(I-\mathsf{E}_n)M_kf
\end{equation*}
and
\begin{equation*}\label{eq:R2}
	R_n^{(2)}f
	=
	-\sum_{k>n}\nu_k\mathsf{E}_n(M_k-I)f.
\end{equation*}
Therefore, by using  the triangle inequality in $L_{p}(\n;\ell_\infty)$ once more, the proof of (\ref{XZmain2}) will be finished if one can show
\begin{align}\label{XZmain4}
	\|{\sup_{n\in\Z}}^+	R_n^{(1)}f\|_{p}\leq C_{d,s}\|{\sup_{k\in\Z}}^+(M_k f^s)^{\frac{1}{s}}\|_{p}.
\end{align}
\begin{align}\label{XZmain5}
	\|{\sup_{n\in\Z}}^+	R_n^{(2)}f\|_{p}\leq C_{d,s}\|{\sup_{k\in\Z}}^+(M_k f^s)^{\frac{1}{s}}\|_{p}.
\end{align}

We first prove (\ref{XZmain4}). Rewrite
$$R_n^{(1)}f=\sum_{j\ge0}\nu_{n-j}(I-\mathsf{E}_n)M_{n-j}f.$$
We claim that for every $j\ge0$ and every positive $f$, there exists a positive constant $C_d$ such that
\begin{align}\label{XZmain6}
	\|
	{\sup_{n\in\Z}}^+(I-\mathsf{E}_n)M_{n-j}f
	\|_{p}
	\le C^{\frac{1}{s'}}_{d}\,2^{-\frac{j}{s'}}\|{\sup_{k\in\Z}}^+(M_k f^s)^{\frac{1}{s}}\|_{p},
\end{align}
where $s'=s/(s-1)$.
Clearly, (\ref{XZmain6}) implies (\ref{XZmain4}). Indeed, by the triangle inequality and noting $\|\nu\|_\infty\leq1$, we get
$$\|{\sup_{n\in\Z}}^+	R_n^{(1)}f\|_{p}\leq\sum_{j\ge0}	 \|
{\sup_{n\in\Z}}^+(I-\mathsf{E}_n)M_{n-j}f
\|_{p}\leq C_{d,s}\|{\sup_{k\in\Z}}^+(M_k f^s)^{\frac{1}{s}}\|_{p}.$$
To prove (\ref{XZmain6}), fix $n,j$ and put $r=2^{-(n-j)}.$ For $x\in\mathbb R^d$, let $Q=Q_n(x)$ be the unique cube in
$\Q_n$ containing $x$. A simple calculation shows that the kernel of $(I-\mathsf{E}_n)M_{n-j}$ is
$$h_{n,j}(x,y)=\frac{1}{|Q||B(0,r)|}\int_Q\big[
\chi_{B(x,r)}(y)-\chi_{B(z,r)}(y)\big]\,dz.$$
Define the positive kernel
\[
a_{n,j}(x,y)
=
\frac{1}{|Q||B(0,r)|}
\int_Q
\left|
\chi_{B(x,r)}(y)-\chi_{B(z,r)}(y)
\right|\,dz
\]
and the associated positive operator
\[
A_{n,j}f(x)
=
\int_{\mathbb R^d}a_{n,j}(x,y)f(y)\,dy.
\]
Then
\begin{equation}\label{eq:order-A}
-A_{n,j}f\le (I-\mathsf{E}_n)M_{n-j}f\le A_{n,j}f.
\end{equation}
If $z\in Q$, then
\[
|x-z|\le\sqrt d\,2^{-n}.
\]
For any $y\in B(x,r)\triangle B(z,r)$, the triangle inequality gives
\[
\big||y-x|-r\big|
\le |x-z|
\le\sqrt d\,2^{-n}.
\]
Thus
\[
B(x,r)\triangle B(z,r)
\subset \Gamma_{n,j}(x),
\]
where
\[
\Gamma_{n,j}(x)
=
\left\{
y\in\mathbb R^d:
\big||y-x|-r\big|\le\sqrt d\,2^{-n}
\right\}.
\]
It follows from above observation and the Fubini theorem that
\begin{equation}\label{eq:A-annulus}
	 A_{n,j}f(x)\le\frac{1}{|B(0,r)|}\int_{\Gamma_{n,j}(x)}f(y)\,dy .
\end{equation}
We have the elementary annular measure estimate
\begin{equation}\label{eq:annular-measure}
	|\Gamma_{n,j}(x)|
	\le C_d r^{d-1}2^{-n}.
\end{equation}
Indeed, set $t=\sqrt d\,2^{-n}$. Then
\[
|\Gamma_{n,j}(x)|
=
\omega_d\big[(r+t)^d-(r-t)_+^d\big],
\]
where $r_+=\max\{r,0\}$ and $\omega_d$ is the volume of the unit ball.
If $t\le r/2$, the mean value theorem gives
\[
(r+t)^d-(r-t)^d\le C_d r^{d-1}t.
\]
If $t>r/2$, then $t\le\sqrt d\,r$ because $2^{-n}\le r$, and hence
\[
(r+t)^d\le C_dr^d\le C_dr^{d-1}t.
\]
This proves \eqref{eq:annular-measure}. Consequently,
\begin{equation}\label{eq:relative-annular-measure}
	\frac{|\Gamma_{n,j}(x)|}{|B(0,r)|}
	\le C_d2^{-j}.
\end{equation}
By the operator-valued H\"older inequality to
\eqref{eq:A-annulus}-Lemma \ref{holderinequality} and using \eqref{eq:relative-annular-measure}, we  deduce that
\begin{align*}
	A_{n,j}f(x)
	&\le
	\left(
	\frac{|\Gamma_{n,j}(x)|}{|B(0,r)|}
	\right)^{1/s'}
	\left(
	\frac1{|B(0,r)|}
	\int_{\Gamma_{n,j}(x)}f(y)^s\,dy
	\right)^{1/s}\\
	&\le
	C^{1/s'}_{d}2^{-j/s'}
	\left(M_{n-j-c_d}(f^s)(x)\right)^{1/s},
\end{align*}
where $c_d$ is a fixed integer chosen so that
\[
\Gamma_{n,j}(x)\subset B(x,2^{c_d}r).
\]
Therefore, we obtain by Remark \ref{rk:MaxFunct} that
$$	\|
{\sup_{n\in\Z}}^+(I-\mathsf{E}_n)M_{n-j}f
\|_{p}
\le\|{\sup_{n\in\Z}}^+	A_{n,j}f
\|_{p}\leq C^{\frac{1}{s'}}_{d}\,2^{-\frac{j}{s'}}\|{\sup_{k\in\Z}}^+(M_k f^s)^{\frac{1}{s}}\|_{p}.$$
This proves (\ref{XZmain6}).

We now turn to the proof of (\ref{XZmain5}). Observe that
$$R_n^{(2)}f=-\sum_{j\ge1}v_{n+j}\mathsf{E}_n(M_{n+j}-I)f.$$
Then it is enough to show for every $j\ge1$ and every positive $f$,
\begin{align}\label{2}
	\|
{\sup_{n\in\Z}}^+\mathsf{E}_n(M_{n+j}-I)f\|_{p}
	\le C^{\frac{1}{s'}}_{d}\,2^{-\frac{j}{s'}}\|{\sup_{k\in\Z}}^+(M_k f^s)^{\frac{1}{s}}\|_{p}.
\end{align}
To see this, fix $n,j$ and put
\[
\delta=2^{-n},
\qquad
\rho=2^{-(n+j)}.
\]
For $x\in Q=Q_n(x)$, the Fubini theorem gives
$$
\mathsf{E}_nM_{n+j}f(x)
=
\frac1{|Q|}
\int_{\mathbb R^d}\theta_Q(y)f(y)\,dy ,
$$
where $\theta_Q(y)$ is given by
\[
\theta_Q(y)
=
\frac{|Q\cap B(y,\rho)|}{|B(0,\rho)|}.
\]
Thus one can write
\[
\mathsf{E}_n(M_{n+j}-I)f(x)
=
\frac1{|Q|}
\int_{\mathbb R^d}
\big(\theta_Q(y)-\chi_Q(y)\big)f(y)\,dy .
\]
Note that the function $\theta_Q-\chi_Q$ vanishes away from the $\rho$-neighborhood
of $\partial Q$.  More precisely, let
\[
U_{n,j}(Q)
=
\left\{
y\in\mathbb R^d:
\operatorname{dist}(y,\partial Q)\le\rho
\right\}.
\]
Then it is not difficult to check that for any $y\in\R^d$,
\begin{align}\label{22}
|\theta_Q(y)-\chi_Q(y)|
\le \chi_{U_{n,j}(Q)}(y).
\end{align}
Define the positive operator
$$
B_{n,j}f(x)
=
\frac1{|Q_n(x)|}
\int_{U_{n,j}(Q_n(x))}f(y)\,dy.
$$
Then it follows from (\ref{22}) that
\begin{equation}\label{eq:order-B}
	-B_{n,j}f
	\le \mathsf{E}_n(M_{n+j}-I)f
	\le B_{n,j}f .
\end{equation}
Observe that $|U_{n,j}(Q)|
\le C_d\rho\delta^{d-1}$, which gives that
\begin{equation}\label{eq:relative-cube-boundary}
	\frac{|U_{n,j}(Q)|}{|Q|}
	\le C_d\frac{\rho}{\delta}
	=C_d2^{-j}.
\end{equation}
Hence, by Lemma \ref{holderinequality} and \eqref{eq:relative-cube-boundary}, we see that
\begin{align*}
	B_{n,j}f(x)
	&\le
	C^{1/s'}_{d}2^{-j/s'}
	\left(
	\frac1{|Q|}
	\int_{U_{n,j}(Q)}f(y)^s\,dy
	\right)^{1/s}.
\end{align*}
Moreover,
\[
U_{n,j}(Q_n(x))
\subset B(x,C_d2^{-n}).
\]
After changing the scale by a fixed dimensional constant, we have
\begin{equation}\label{eq:B-pointwise}
	B_{n,j}f(x)
	\le C^{1/s'}_{d}2^{-j/s'}
	\left(M_{n-c_d}(f^s)(x)\right)^{1/s}.
\end{equation}
Therefore,  by (\ref{eq:order-B}), (\ref{eq:B-pointwise}) and Remark \ref{rk:MaxFunct}, we obtain (\ref{2}). This finishes the proof of the proposition.
\end{proof}

With Proposition \ref{cotlar1} in hand, we can present the proof of strong type $(p,p)$ estimates of $(T_N)_{N\in\Z^2_{<}}$ stated in Theorem \ref{maintheorem}.
\begin{proof}[Proof of Theorem \ref{maintheorem} (i).]
As we discussed in the proof of Proposition \ref{cotlar1}, we may assume that $f$ is positive. Now fix $1<p<\infty$ and $f\in L_p(\n)_+$.  By  Proposition \ref{cotlar1}, it suffices to estimate the two terms on the right hand side of (\ref{cotlar2}). To that end, we first use the noncommutative Doob's maximal inequality \cite{J1} to deduce that
\begin{equation}\label{eq1}
	\|{\sup_{n\in\Z}}^+\mathsf{E}_{n}(Tf)\|_{p}\leq c_{d,p}\|Tf\|_p\leq C_{d,p}\|f\|_p,
\end{equation}
where the second inequality follows from Lemma \ref{xu2023}.

It remains to estimate the second term on the  right hand side of (\ref{cotlar2}). Choose $s$ satisfying $1<s<p$, with the specific value to be determined later. By using the noncommutative Hardy-Littlewood
maximal inequalities  \cite{M}  to $f^s\in L_{p/s}(\n)$, we may find a positive operator $F\in L_{p/s}(\n)$ such that for any $k\in\Z$,
$$\|F\|_{p/s}\leq C_d\frac{p^2}{(p-s)^2}\|f^s\|_{p/s}\ \mbox{and}\ M_k f^s\leq F.$$
Consequently, by the monotone increasing property of $0\leq a\mapsto a^t$ for $0<t<1$, we see that for any $k\in\Z$,
$$(M_k f^s)^{1/s}\leq F^{1/s}\ \mbox{and}\ \|F^{1/s}\|_p=\|F\|^{1/s}_{p/s}\leq C^{1/s}_d\frac{p^{2/s}}{(p-s)^{2/s}}\|f\|_p.$$
This gives by Remark \ref{rk:MaxFunct} that
\begin{equation}\label{eq2}
\|{\sup_{k\in\Z}}^+(M_k f^s)^{\frac{1}{s}})\|_{p}\leq\|F^{1/s}\|_p\leq C^{1/s}_d\frac{p^{2/s}}{(p-s)^{2/s}}\|f\|_p.
\end{equation}

Combining with Proposition \ref{cotlar1}, (\ref{eq1}) and (\ref{eq2}), we obtain
$$\|{\sup_{N\in\Z^2_{<}}}^+T_{N}f\|_p\leq \big(C_{d,p,s}+C_{d,s}\frac{p^{2/s}}{(p-s)^{2/s}}\big)\|f\|_p.$$
By letting $s={1+p\over2}$,  
we finally have
$$\|{\sup_{N\in\Z^2_{<}}}^+T_{N}f\|_p\leq C_{d,p}\|f\|_p.$$
This completes the proof.
\end{proof}


\section{Proof of Theorem \ref{maintheorem}: weak type $(1,1)$ estimate}\label{Sec:4}
This section is devoted to the proof of endpoint estimate of $(T_N)_{N\in\Z^2_{<}}$ stated in Theorem \ref{maintheorem}.
\subsection{Some reductions}
Without loss of generality, we may assume that $f$ is positive and the sequence $(\nu_k)$ is real. Moreover,
since $\mathcal N_{c,+}$ is dense in the positive cone of $L_1(\mathcal N)$, a standard approximation argument shows that it suffices to establish the desired weak type $(1,1)$ estimate for $f\in \mathcal N_{c,+}$.

Now fix $f\in\mathcal N_{c,+}$ and $\lambda>0$. By the definition of $\Lambda_{1,\infty}(\n;\ell_\infty(\Z^2_{<}))$, it is enough to show that there exists a projection $e\in\n$ such that
\begin{align}\label{XZ4}
	\sup_{N\in\Z^2_{<}}\|eT_{N}f e\|_{\infty}\leq\lambda\ \ \mbox{and}\ \ \varphi(e^{\perp})\le C_d\frac{\| f\|_{1}}{\lambda}.
\end{align}

By translating the function $f$ (i.e., shifting the index of the dyadic filtration
by $m_{\lambda}(f)$), we may assume that
$m_{\lambda}(f)=0$, where $m_{\lambda}(f)$ is the integer provided by
Lemma~\ref{cucu}.  This translation corresponds to re-indexing the operators
$M_k$ and $\mathsf E_k$ by $k\mapsto k+m_{\lambda}(f)$, which does not affect
any of the norm estimates since the family $(T_N)_N$ is invariant under such
a simultaneous shift of the index set $\Z^2_{<}$. Now using Theorem \ref{czdecom}, we can write $f=g+b+h$. Thus it suffices to find a projection $e\in \mathcal N$ such that
\begin{align*}
	\forall N\in\Z^2_{<},\;\max\Big\{\|eT_{N}g e\|_\infty,\;\|eT_{N}b e\|_\infty,\;\|eT_{N}h e\|_\infty\Big\}\leq\lambda\ \ \text{and}
	\ \ \varphi(e^{\perp})\lesssim\frac{\|f\|_1}{\lambda},
\end{align*}
which will follow from, by setting $e=e_1\wedge e_2\wedge e_3$, the existences of three projections $e_1,e_2$ and $e_3$ such that
\begin{align}\label{XZ5}
	\sup_{N\in\Z^2_{<}}\|e_{1}T_{N}g e_{1}\|_{\infty}\leq\lambda\ \ \mbox{and}\ \ \varphi(e_{1}^{\perp})\lesssim\frac{\| f\|_{1}}{\lambda},
\end{align}
\begin{align}\label{XZ6}
	 \sup_{N\in\Z^2_{<}}\|e_{2}T_Nbe_{2}\|_{\infty}\leq\lambda\ \ \mbox{and}\ \ \varphi(e_{2}^{\perp})\lesssim\frac{\| f\|_{1}}{\lambda},
\end{align}
and
\begin{align}\label{XZ7}
	 \sup_{N\in\Z^2_{<}}\|e_{3}T_Nhe_{3}\|_{\infty}\leq\lambda\ \ \mbox{and}\ \ \varphi(e_{3}^{\perp})\lesssim\frac{\| f\|_{1}}{\lambda}.
\end{align}
\subsection{Estimate of (\ref{XZ5})}
The desired maximal weak type $(1,1)$ estimate for the good function is a consequence of Theorem \ref{maintheorem} (i).
\begin{proof}[Proof of \eqref{XZ5}.]
Observe that $g$ is positive according to Theorem \ref{czdecom} (ii) and each $T_Ng$ is selfadjoint. Applying Theorem \ref{maintheorem} (i) and Remark \ref{rk:MaxFunct}, there exists a positive operator $a\in L_{2}(\mathcal{N})$ such that for any $N\in\Z^2_{<}$,
	$$-a\leq T_{N}g\leq a\ \ \ \mbox{and}\ \ \ \|a\|_{2}\lesssim\|g\|_{2}.$$
	Set $e_{1}=\chi_{(0,\lambda]}(a)$. Then for any $N\in\Z^2_{<}$,
	$$-\lambda\leq-e_{1}ae_{1}\leq e_{1}T_{N}ge_{1}\leq e_{1}ae_{1}\leq\lambda.$$
Furthermore, by the Chebyshev and H\"{o}lder inequalities,
	 $$\varphi(e_{1}^{\perp})=\varphi(\chi_{(\lambda,\infty)}(a))\leq
	 \frac{\|a\|^{2}_{2}}{\lambda^{2}}\lesssim\frac{\|g\|^{2}_{2}}{\lambda^{2}}\leq \frac{\|g\|_1\|g\|_{\infty}}{\lambda^{2}}\lesssim\frac{\|f\|_{1}}{\lambda},$$
	where the last inequality follows from  Theorem \ref{czdecom} (ii). This, by Remark \ref{weakm4}, we obatin the desired estimate for $(T_{N}g)_{N\in\Z^2_{<}}$.
\end{proof}
\subsection{Estimate of (\ref{XZ6})}
Using the projection $\zeta$ constructed in (\ref{czd9}), we decompose $T_{N}b$ in the following way $$T_{N}b= \zeta^\perp T_{N}b \zeta^\perp + \zeta \hskip1pt T_{N}b\zeta^\perp
+ \zeta^\perp T_{N}b\zeta + \zeta \hskip1pt T_{N}b
\zeta.$$
Then taking $e_{2,1}=\zeta$, we are reduced to finding a projection $e_{2,2}$ such that
\begin{align}\label{bd}
	\sup_{N\in\Z^2_{<}}\|e_{2,2}\zeta T_{N}b\zeta e_{2,2}\|_{\infty}\leq\lambda\ \ \mbox{and}\ \ \varphi(e_{2,2}^{\perp})\lesssim\frac{\| f\|_{1}}{\lambda}.
\end{align}
Indeed, taking $e_2=e_{2,1}\wedge e_{2,2}$, we obtain for any $N\in\Z^2_{<}$
$$\|e_{2}T_{N}be_{2}\|_\infty\leq \|e_{2,2}\zeta T_{N}b\zeta e_{2,2}\|_{\infty}\leq\lambda.$$
Together with (\ref{bd}) and Theorem \ref{czdecom} (i), we get
$$\varphi(e_{2}^{\perp})\leq\varphi(e^{\perp}_{2,1})+\varphi(e^{\perp}_{2,2})\lesssim\frac{\| f\|_{1}}{\lambda}.$$
Thus we get estimate (\ref{XZ6}).

Now we are at a position to show \eqref{bd}.
\begin{proof}[Proof of \eqref{bd}.]
	Recall that $m_{\lambda}(f)=0$, we have $p_n=0$ for all $n\leq 0$. By Theorem \ref{czdecom} (iii),  $b=\sum_{n=1}^{\infty}b_{n}$ where $b_{n}=p_n(f-f_n)p_n$ as in (\ref{czd6}).
	We first claim that for  $k\geq n$,
	\begin{eqnarray}\label{diff}
		\zeta(M_k-\mathsf{E}_{k}) \hskip-1pt b_{n}\zeta=0.
	\end{eqnarray}
	Indeed, fix one $x\in\mathbb R^d$. For $k\ge n$, the ball $B_k$ has radius $2^{-k}\le 2^{-n}$.
	Since $x\in Q_{x,n}$,
	the distance from $x$ to the boundary of $5Q_{x,n}$ in any direction is at least $2\cdot2^{-n}\ge2\cdot2^{-k}>2^{-k}$.
	Hence $x+B_k\subset 5Q_{x,n}$.
	By the cancellation property in Theorem~\ref{czdecom} (iii), we have $\zeta(x)b_n(y)\zeta(x)=0$ for every $y\in 5Q_{x,n}$.
	Consequently,
	$$
	\zeta(x)M_k b_n(x)\zeta(x)
	= \frac{1}{|B_k|}\int_{x+B_k} \zeta(x)b_n(y)\zeta(x)\,dy = 0.
	$$
	Similarly, for $k\ge n$, the dyadic cube $Q_{x,k}$ satisfies $Q_{x,k}\subset Q_{x,n}\subset 5Q_{x,n}$,
	and therefore
	$$
	\zeta(x)\mathsf{E}_k b_n(x)\zeta(x)
	= \frac{1}{|Q_{x,k}|}\int_{Q_{x,k}} \zeta(x)b_n(y)\zeta(x)\,dy = 0.
	$$
	This establishes the claim (\ref{diff}).
	On the other hand, for $k< n$, by applying the property of conditional expectation, one has $\mathsf{E}_k b_{n}=\mathsf{E}_k\mathsf{E}_nb_{n}=0$.

	Taking these observations into consideration, we deduce that for any $x\in\mathbb R^d$,
	\begin{align*}
		\zeta(x)T_{N}b(x)\zeta(x)
		 =\zeta(x)\sum_{n=1}^{\infty}\sum_{{\begin{subarray}{c}
				k:N_1\leq k\leq N_2 \\ k<n
			\end{subarray}}}\nu_kM_{k}b_{n}(x)\zeta(x).
	\end{align*}
	Furthermore, by the cancellation property of $b_{n}$ stated in Theorem \ref{czdecom} (iii), it is easy to verify that
	\begin{align}\label{bad3422237289}
		 M_kb_{n}(x)=\frac{1}{|B_{k}|}\int_{x+B_{k}}b_{n}(y)dy=\frac{1}{|B_{k}|}\sum_{{\begin{subarray}{c}
					Q\in\Q_{n} \\ (x+\partial B_{k})\cap Q\neq \emptyset
		\end{subarray}}}\int_{(x+B_{k})\cap Q}b_{n}(y)dy,
	\end{align}
	where $\partial B$ denotes the boundary of the ball $B$.

	We now handle the operator-valued integrand $b_n(y)$, which is not necessarily positive. For each $x\in\R^d$, $n\in\N$, $k<n$, and $Q\in\Q_n$ with $(x+\partial B_k)\cap Q\neq\emptyset$, define the selfadjoint operator
	$$
	A_{n,k,Q}(x)=\frac{1}{|B_k|}\int_{(x+B_k)\cap Q}b_n(y)\,dy\in\M.
	$$
	Recall that $|a|=(a^*a)^{1/2}$. Define
	\begin{equation}\label{F1def}
	F_1(x)=\sum_{n=1}^{\infty}\sum_{k<n}\;
		\Bigl|\nu_k\sum_{\substack{Q\in\Q_n\\ (x+\partial B_k)\cap Q\neq\emptyset}}
A_{n,k,Q}(x)\Bigr|.
	\end{equation}
	Since $-|a|\le a\le |a|$ for any selfadjoint $a\in\M$, and $F_1(x)$ is positive by construction, we obtain the valid operator inequality
	\begin{align}\label{F1}
		-\zeta F_{1}\zeta\leq \zeta T_{N}b\zeta\leq \zeta F_{1}\zeta.
	\end{align}
To complete the proof of \eqref{bd}, it is enough to show
	\begin{align}\label{F121}
		\|F_{1}\|_{1}\lesssim\|f\|_{1}.
	\end{align}
Indeed, with (\ref{F121}) in hand, set $e_{2,2}=\chi_{(0,\lambda]}(\zeta F_{1}\zeta)$. Then \eqref{F1} implies for any $N\in\Z^2_{<}$
	\begin{align*}
		-\lambda\leq-e_{2,2}\zeta F_{1}\zeta e_{2,2}\leq e_{2,2}\zeta T_{N}b\zeta e_{2,2}\leq e_{2,2}\zeta F_{1}\zeta e_{2,2}\leq\lambda;
	\end{align*}
	moreover, by the Chebyshev inequality and (\ref{F121}), we find
	$$ \varphi(e_{2,2}^{\perp})\leq\frac{\|\zeta F_{1}\zeta\|_{1}}{\lambda}\leq\frac{\| F_{1}\|_{1}}{\lambda}\lesssim\frac{\| f\|_{1}}{\lambda}.$$
	This gives the desired estimate (\ref{bd}) thanks to Remark \ref{weakm4}.

	It remains to show (\ref{F121}). 	For $k<n$, denote
	 $$M_{k,n}b_n(x)=\frac{1}{|B_{k}|}\sum_{{\begin{subarray}{c}
				Q\in\Q_{n} \\ (x+\partial B_{k})\cap Q\neq \emptyset
	\end{subarray}}}\int_{(x+B_{k})\cap Q}b_{n}(y)dy.$$
	Then by definition,
	$$M_{k,n}b_n(x)=\sum_{\substack{Q\in\Q_n\\ (x+\partial B_k)\cap Q\neq\emptyset}}
	A_{n,k,Q}(x).$$	
On the other hand, 	by \cite[Lemma 3.8]{HX}, one has
	\begin{align*}
		\|M_{k,n}b_n\|_1\lesssim2^{k-n}\|b_n\|_1.
	\end{align*}
	
Now by the definition of $F_1$ in (\ref{F1def}), the identity $\||a|\|_1=\|a\|_1$ for $a\in L_1(\M)$, $\|\nu\|_\infty\leq1$ and the triangle inequality for the $L_1$-norm, we obtain
	 $$\|F_1\|_1\lesssim\sum_{n=1}^{\infty}\sum_{k:k<n}2^{k-n}\|b_n\|_1\leq\sum_{n=1}^{\infty} \|b_n\|_1\lesssim\|f\|_1,$$
	which finishes the proof.
\end{proof}
\subsection{Estimate of (\ref{XZ7})}
Let us now turn to the last term $h$. By taking $e_{3,1}=\zeta$ and
using the same argument as in estimating $T_{N}b$, we are reduced to finding a  projection $e_{3,2}\in\mathcal{N}$ such that
\begin{align}\label{bd11}
		\sup_{N\in\Z^2_{<}}\|e_{3,2}\zeta T_{N}h\zeta e_{3,2}\|_{\infty}\leq\lambda\ \ \mbox{and}\ \ \varphi(e_{3,2}^{\perp})\lesssim\frac{\| f\|_{1}}{\lambda}.
\end{align}
\begin{proof}[Proof of \eqref{bd11}.]
	By noting that $m_{\lambda}(f)=0$, we may write $h$ as $h=\sum_{n=1}^{\infty}h_{n}$ with $h_{n}=p_n (f-f_{n}) q_n+q_{n}(f-f_{n})p_n$ as in (\ref{czd66}).
Moreover, by Theorem \ref{czdecom} (iv) and applying an argument similar to that applied to the term  $b$ (see \eqref{diff}), we can deduce that $\zeta(M_k-\mathsf{E}_{k}) \hskip-1pt h_{n}\zeta=0$ whenever $k\geq n$; and for $k< n$, $\mathsf{E}_k h_{n}=\mathsf{E}_k\mathsf{E}_nh_{n}=0$.
	Consequently, for $x\in\mathbb R^d$,
	\begin{align*}
		\zeta(x)T_{N}h(x)\zeta(x)
		& =\zeta(x)\sum_{n=1}^{\infty}\sum_{{\begin{subarray}{c}
					k:N_1\leq k\leq N_2 \\ k<n
		\end{subarray}}}\nu_kM_{k}h_{n}(x)\zeta(x).
	\end{align*}
	It then follows from the cancellation property stated in Theorem \ref{czdecom} (iv) that
	\begin{align}\label{termh1}
		 M_kh_{n}(x)=\frac{1}{|B_{k}|}\int_{x+B_{k}}h_{n}(y)dy=\frac{1}{|B_{k}|}\sum_{{\begin{subarray}{c}
					Q\in\Q_{n} \\ (x+\partial B_{k})\cap Q\neq \emptyset
		\end{subarray}}}\int_{(x+B_{k})\cap Q}h_{n}(y)dy.
	\end{align}
	As in the $b$ estimate, we handle the non-positive integrand via the absolute value in $\M$. For each $x$, $n$, $k<n$, and $Q$, define
	$$
	B_{n,k,Q}(x)=\frac{1}{|B_k|}\int_{(x+B_k)\cap Q}h_n(y)\,dy\in\M,
	$$
	and set
	$$
	F_2(x)=\sum_{n=1}^{\infty}\sum_{k<n}\;
	\Bigr|\nu_k\sum_{\substack{Q\in\Q_n\\ (x+\partial B_k)\cap Q\neq\emptyset}}
	B_{n,k,Q}(x)\Bigr|.
	$$
	This gives that for any $N\in\Z^2_{<}$,
	\begin{align}\label{F2}
	-\zeta F_{2}\zeta\leq \zeta T_{N}h\zeta\leq \zeta F_{2}\zeta.
\end{align}
	In the following, we show that
	\begin{align}\label{F1211}
		\|F_{2}\|_{1}\lesssim\|f\|_{1}.
	\end{align}
	Observe that
	\begin{align}\label{termh}
		h_{n}=p_n f q_n+q_{n}fp_n.
	\end{align}
	Indeed, by the definition of $p_{n}$-(\ref{czd1}), $p_{n}=q_{n-1}-q_{n}\leq q_{n-1}$. Hence, by Lemma \ref{cucu} (i),
	$$p_n f_{n} q_n=p_nq_{n-1} f_{n} q_{n-1}q_n=p_nq_nq_{n-1} f_{n} q_{n-1}=0.$$
	The same reasoning gives that $q_nf_np_n=0$. This proves (\ref{termh}).

  Now we set
  \begin{eqnarray*}
  	\|G_{1,n}\|_{1}& \triangleq &\sum_{k:k<n}\int_{\R^{d}}\Big\|\frac{1}{|B_{k}|}\sum_{{\begin{subarray}{c}
  				Q\in\Q_{n} \\ (x+\partial B_{k})\cap Q\neq \emptyset
  	\end{subarray}}}\int_{(x+B_{k})\cap Q}(p_n f q_n)(y)dy\Big\|_{L_{1}(\M)}dx,\\
  	\|G_{2,n}\|_{1}& \triangleq &\sum_{k:k<n}\int_{\R^{d}}\Big\|\frac{1}{|B_{k}|}\sum_{{\begin{subarray}{c}
  				Q\in\Q_{n} \\ (x+\partial B_{k})\cap Q\neq \emptyset
  	\end{subarray}}}\int_{(x+B_{k})\cap Q}(q_{n}fp_n)(y)dy\Big\|_{L_{1}(\M)}dx.
  \end{eqnarray*}
  Then by observations (\ref{termh}), (\ref{termh1}) and the Minkowski inequality, we conclude that to obtain  (\ref{F1211}), it suffices to show
  \begin{align}\label{bd1111}  	 \max\Big\{\sum_{n=1}^{\infty}\|G_{1,n}\|_{1},\sum_{n=1}^{\infty}\|G_{2,n}\|_{1}\Big\}\lesssim\|f\|_1.
  \end{align}
  By symmetric, we just consider $\|G_{1,n}\|_{1}$. We claim that
  \begin{align}\label{bd111}
   \|G_{1,n}\|_{1}\lesssim\Big(\lambda\varphi(p_{n})\Big)^{\frac{1}{2}}\Big(\varphi(fp_{n})\Big)^{\frac{1}{2}}.
  \end{align}

To see this, note that for a given point $y$, there exists a unique $Q\in\Q_n$ containing $y$ such that $p_n(y)=p_Q$ and $q_n(y)=q_Q$. Then by the H\"older inequality in the noncommutative  Hilbert-valued $L_p$-spaces (see e.g. \cite[Proposition 1.2]{M}), we have
\begin{align*}\label{diff1}
		&\quad\Big\|\int_{(x+B_k)\cap Q}p_Q f(y) q_Qdy\Big\|_{L_{1}(\M)}\\
		& \leq\Big\|\Big(\int_{(x+B_k)\cap Q}p_Qf(x)p_Qdx\Big)^{\frac{1}{2}}\Big\|_{L_{1}(\M)}
		\Big\|\Big(\int_{(x+B_k)\cap Q}q_Qf(x)q_Qdx\Big)^{\frac{1}{2}}\Big\|_{L_{\infty}(\M)}\\& \leq\Big\|\Big(\int_{ Q}p_Qf(x)p_Qdx\Big)^{\frac{1}{2}}\Big\|_{L_{1}(\M)}
		 \Big\|\Big(\int_{Q}q_Qf(x)q_Qdx\Big)^{\frac{1}{2}}\Big\|_{L_{\infty}(\M)}\\
		& =\Big\|\int_{ Q}p_Qf(x)p_Qdx\Big\|^{\frac{1}{2}}_{L_{\frac{1}{2}}(\M)}
		 \Big\|\int_{Q}q_Qf(x)q_Qdx\Big\|^{\frac{1}{2}}_{L_{\infty}(\M)}.
\end{align*}
Applying the H\"{o}lder inequality to the first term and Lemma \ref{cucu} (ii) to the second term above, we deduce that
\begin{align*}
	\hskip1pt \Big\|\int_{(x+B_k)\cap Q}p_Q f(y) q_Qdy\Big\|_{L_{1}(\M)}
	\leq&\ \|p_{Q}\|^{\frac{1}{2}}_{L_{1}(\M)}
	 \Big\|\int_{Q}p_Qf(x)p_Qdx\Big\|^{\frac{1}{2}}_{L_{1}(\M)}(\lambda|Q|)^{\frac{1}{2}}\\
	=&\ \left[\varphi(fp_{Q}\chi_{Q})\tau(p_{Q})\lambda|Q|\right]^{\frac{1}{2}}.
\end{align*}
Putting above estimate into the definition of $G_{1,n}$, using the Fubini theorem and $|B_k|\approx 2^{-kd}$, we find
\begin{align*}
\hskip1pt \|G_{1,n}\|_{1}
\lesssim&\ \hskip1pt \sum_{k:k<n}2^{kd}\int_{\R^{d}}\sum_{{\begin{subarray}{c}Q\in\Q_{n} \\ (x+\partial B_{k})\cap Q\neq \emptyset\end{subarray}}}(\varphi(fp_{Q}\chi_{Q})\tau(p_{Q})\lambda|Q|)^{\frac{1}{2}}dx\\
=&\ \sum_{k:k<n}2^{kd}\sum_{Q\in\Q_{n}}
(\varphi(fp_{Q}\chi_{Q})\tau(p_{Q})\lambda|Q|)^{\frac{1}{2}}\int_{\R^{d}}\chi_{(x+\partial B_{k})\cap Q\neq \emptyset}dx.
\end{align*}
Observe that
$$\int_{\R^{d}}\chi_{(x+\partial B_{k})\cap Q\neq \emptyset}dx\lesssim2^{-k(d-1)}2^{-n}.$$
Then  the Cauchy-Schwarz inequality gives that
\begin{align*}
\|G_{1,n}\|_{1}
\lesssim&\ \hskip1pt \sum_{k:k<n}\sum_{Q\in\Q_{n}}2^{k-n}(\varphi(fp_{Q}\chi_{Q})\tau(p_{Q})\lambda|Q|)^{\frac{1}{2}}\\
\lesssim&\ \sum_{Q\in\Q_{n}}(\varphi(fp_{Q}\chi_{Q})\tau(p_{Q})\lambda|Q|)^{\frac{1}{2}}\\
\leq&\ \Big(\sum_{Q\in\Q_{n}}\tau(p_{Q})
|Q|\lambda\Big)^{\frac{1}{2}}\Big(\sum_{
Q\in\Q_{n}}\varphi(fp_{Q} \chi_{Q})\Big)^{\frac{1}{2}}\\
=&\ \Big(\lambda\varphi(p_{n})\Big)^{\frac{1}{2}}\Big(\varphi(fp_{n})\Big)^{\frac{1}{2}}.
\end{align*}
This proves the claim (\ref{bd111}).

Now using (\ref{bd111}) and the Cauchy-Schwarz inequality once more, we have
\begin{align*}
 \sum_{n=1}^{\infty}\|G_{1,n}\|_1
  & \lesssim\Big(\sum_{n=1}^{\infty}\lambda\varphi(p_{n})\Big)^{\frac{1}{2}}
  \Big(\sum_{n=1}^{\infty}\varphi(fp_{n})\Big)^{\frac{1}{2}}\\
  & =  \Big(\lambda\varphi(1-q)\Big)^{\frac{1}{2}}
 \Big(\varphi(f(1-q))\Big)^{\frac{1}{2}}\lesssim\|f\|_{1},
\end{align*}
 where the last inequality follows from Lemma \ref{cucu} (iii).
This completes the proof of (\ref{bd1111}).
 			
Finally,  set $e_{3,2}=\chi_{(0,\lambda]}(\zeta F_{2}\zeta)$. Then by (\ref{F2}) and (\ref{bd1111}), we get for any $N\in\Z^2_{<}$,
 $$\|e_{3,2}\zeta T_{N}h\zeta e_{3,2}\|_{\infty}\leq\lambda\ \ \mbox{and}\ \
\varphi(e_{3,2}^{\perp})\leq\frac{\|F_{2}\|_{1}}{\lambda}\lesssim\frac{\| f\|_{1}}{\lambda}.$$
This completes the proof of (\ref{bd11}).
 \end{proof}

\subsection{Conclusion}
Combining the results obtained so far in Subsection 4.1-4.4, we obtain the maximal weak type $(1,1)$ estimate of $(T_N)_N$ announced in Theorem \ref{maintheorem}. 

\medskip

\section{Proof of Theorem \ref{maintheorem}: b.a.u. convergence}\label{Sec:convergence}

To complete the proof of Theorem \ref{maintheorem}, we provide the argument for b.a.u.\
convergence. Recall that the index set $\Z^2_{<}$ forms a net. More precisely, for $N=(N_1,N_2)$ and $N'=(N_1',N_2')$, we denote
$N\preceq N'$ if $N_1'\le N_1$ and $N_2\le N_2'$. Hence, the usual sequential reasoning must be adapted to this directed setting.
\begin{lem}\label{denseset}
	Define
	$$\mathcal D = \left\{
	\sum^m_{i=1}a_i \otimes \chi_{Q_i}
	\;:\;a_i\in \mathcal{S_{\M}},\;Q_i\in\mathcal Q,\;m\in\N
	\right\}.$$
	Then for any $h\in\mathcal D$, 
	\begin{equation*}
		\|T_{N}h - T_{N'}h\|_\infty\rightarrow0\ \mbox{as}\ N_1\rightarrow-\infty,N_2\rightarrow\infty
	\end{equation*}
where $N=(N_1,N_2)$ and $N'=(N_1',N_2')$ with $N\preceq N'$.
\end{lem}
\begin{proof}
	For $h=\sum_{i=1}^m a_i\otimes \chi_{Q_i}\in\mathcal D$, we have for every $k\in\mathbb Z$,
	$$
	M_k h(x)=\sum_{i=1}^m a_i\,\frac{|(x+B_k)\cap Q_i|}{|B_k|},
	\qquad
	\mathsf E_k h(x)=\sum_{i=1}^m a_i\,\frac{|Q_{x,k}\cap Q_i|}{|Q_{x,k}|},
	$$
	where $Q_{x,k}$ is the unique dyadic cube of side-length $2^{-k}$ containing $x$.
	
	We claim that there exist constants $C_h>0$ and $K_0\in\mathbb Z$ such that
	\begin{equation}\label{xu-zhang112}
		\|(M_k-\mathsf E_k)h\|_\infty \le C_h \min\{2^{kd},\,2^{-k}\}
		\quad\text{for all } k\in\mathbb Z.
	\end{equation}
Indeed, for $k\to -\infty$, each $Q_i$ is contained in at most $C_d$ dyadic cubes of side-length $2^{-k}$ and in at most $C_d$ balls of radius $2^{-k}$; hence
$$
\|(M_k-\mathsf E_k)(a_i\otimes \chi_{Q_i})\|_\infty
\le \|a_i\|_{\mathcal M}\, |Q_i|\, (|B_k|^{-1}+|Q_{x,k}|^{-1})
\le C_{d,i}\|a_i\|_{\mathcal M}\,2^{kd}.
$$
On the other hand, it is not difficult to see there exists $C_h>0$ such that $\|(M_k-\mathsf E_k)h\|_\infty\le C_h 2^{-k}$ for all sufficiently large $k$. This proves (\ref{xu-zhang112}).

Consequently, for $N\preceq N'$,
\[
\begin{aligned}
	T_{N'}h-T_Nh
	&=\sum_{k=N_1'}^{N_1-1}\nu_k(M_k-\mathsf E_k)h
	+\sum_{k=N_2+1}^{N_2'}\nu_k(M_k-\mathsf E_k)h.
\end{aligned}
\]
Since $\|\nu\|_{\ell_\infty}\le1$, using (\ref{xu-zhang112}) we obtain
\[
\begin{aligned}
	\|T_{N'}h-T_Nh\|_\infty
	&\le \sum_{k=N_1'}^{N_1-1}\|(M_k-\mathsf E_k)h\|_\infty
	+\sum_{k=N_2+1}^{N_2'}\|(M_k-\mathsf E_k)h\|_\infty \\
	&\le C_h\left(\sum_{k=N_1'}^{N_1-1}2^{kd}+\sum_{k=N_2+1}^{N_2'}2^{-k}\right) \\
	&= C_h\left(2^{d(N_1-1)}-2^{d(N_1'-1)} + 2^{-N_2}-2^{-N_2'}\right) \to 0
\end{aligned}
\]
as $N_1\to -\infty$ and $N_2\to\infty$. The proof is finished.
\end{proof}

\begin{proof}[Proof of Theorem \ref{maintheorem} (iii)]
Let $1\leq p<\infty$ and $f\in L_p(\n)$. By  \cite[Proposition 1.3]{CL}, it  is enough  to show $(T_{N}f)_{N\in\Z^2_{<}}$ is a Cauchy net. More precisely, for any $\varepsilon>0$, there exists a projection $e\in\n$ such that
\begin{equation}\label{eq:tail}
	\varphi(1-e)<\varepsilon\ \mbox{and}\ \|e(T_{N}f - T_{N'}f)e\|_\infty\rightarrow0\ \mbox{as}\ N_1\rightarrow-\infty,N_2\rightarrow\infty
\end{equation}
where $N=(N_1,N_2)$ and $N'=(N_1',N_2')$ with $N\preceq N'$. Since $\mathcal D$ is dense in $L_p(\n)$, we may find $g_n\in\mathcal D$ such that $\|f-g_n\|_p^p\leq\frac{\varepsilon}{2^{n+1}n^p}$. Then using the fact that $(T_N)_N$ is of weak type $(p,p)$, there is a projection $e_n\in\n$ so that for any $N\in\Z^2_{<}$,
$$\|e_n(T_N(f-g_n))e_n\|_\infty\leq\frac{1}{n}\ \mbox{and}\ \varphi(1-e_n)\leq n^p\|f-g_n\|_p^p\leq\frac{\varepsilon}{2^{n+1}}.$$
Set $e=\wedge_ne_n$. Then
\begin{equation}\label{eq:ta}
	\varphi(1-e)<\frac{\varepsilon}{2}\ \mbox{and}\ \sup_{N\in\Z^2_{<}}\|e(T_N(f-g_n))e\|_\infty\leq\frac{1}{n}\ \forall n\geq1.
\end{equation}

On the other hand, by Lemma \ref{denseset}, for each fixed $n$, there exists $K_n>0$ such that whenever $N_1\le -K_n$ and $N_2\ge K_n$,
$$
\|T_{N'}g_n-T_Ng_n\|_\infty < \frac{1}{n}.
$$

By above estimates,
for $N\preceq N'$ with $N_1\le -K_n$ and $N_2\ge K_n$, we have
\begin{align*}
	\|e(T_{N'}f-T_Nf)e\|_\infty
	&\le \|e(T_{N'}f-T_{N'}g_n)e\|_\infty
	+\|T_{N'}g_n-T_Ng_n\|_\infty \\
	&\quad + \|e(T_Ng_n-T_Nf)e\|_\infty \\
	&\le \frac{1}{n} + \frac{1}{n} + \frac{1}{n}
	= \frac{3}{n}.
\end{align*}
Now let $N_1\rightarrow-\infty,N_2\rightarrow\infty$ firstly and later $n\rightarrow\infty$, the above inequality yields
$$\lim_{N_1\rightarrow-\infty,N_2\rightarrow\infty}\|e(T_{N}f - T_{N'}f)e\|_\infty=0.$$
This completes the proof.
\end{proof}


\end{document}